\documentclass[letterpaper]{amsart}

\usepackage{amsmath}
\usepackage{amsfonts}
\usepackage{amssymb}
\usepackage{amsthm}
\usepackage{array}
\usepackage{graphicx}
\usepackage{tikz-cd}
\usepackage{setspace}
\usepackage{enumerate}
\usepackage{hyperref}

\newtheorem{thm}{Theorem}[section]
\newtheorem{lem}[thm]{Lemma}
\newtheorem{cor}[thm]{Corollary}
\newtheorem{prop}[thm]{Proposition}

\theoremstyle{definition}

\theoremstyle{remark}
\newtheorem*{remark}{Remark}

\newtheorem*{acknowledgements}{Acknowledgements}

\DeclareMathOperator{\id}{id}

\DeclareMathOperator{\Mod}{mod}

\DeclareMathOperator{\re}{Re}
\DeclareMathOperator{\im}{Im}

\DeclareMathOperator{\Aut}{Aut}

\DeclareMathOperator{\length}{length}

\DeclareMathOperator{\teich}{Teich}
\DeclareMathOperator{\axis}{axis}

\newcommand{\R}{\mathbb{R}}
\newcommand{\D}{\mathbb{D}}
\newcommand{\Co}{\mathbb{C}}
\newcommand{\CHAT}{\widehat{\mathbb{C}}}
\newcommand{\Z}{\mathbb{Z}}
\newcommand{\N}{\mathbb{N}}
\newcommand{\Sp}{\mathbb{S}}

\begin{document}

\title[Convergence of twist coordinates]{Conformal grafting and convergence of Fenchel-Nielsen twist coordinates}

\author{Maxime Fortier Bourque}
\thanks{Research partially supported by the Natural Sciences and Engineering Research Council of Canada.}
\address{Department of Mathematics, The Graduate Center, City University of New York, New York, NY, USA}
\email{maxforbou@gmail.com}


\begin{abstract}
We cut a hyperbolic surface of finite area along some analytic simple closed curves, and glue in cylinders of varying moduli. We prove that as the moduli of the glued cylinders go to infinity, the Fenchel-Nielsen twist coordinates for the resulting surface around those cylinders converge.
\end{abstract}

\maketitle

\tableofcontents

\section{Introduction}

Let $S$ be a hyperbolic Riemann surface with finite area. We want to perform a surgery on $S$ which we call \textit{conformal grafting}. Let $E$ be an analytic multicurve in $S$, i.e. a set of disjoint, non-parallel, non-peripheral, essential, simple, closed, parametrized analytic curves. For every vector $t\in (\R_{\geq 0})^E$, we construct a new surface $S_t$ as follows. We first cut $S$ along $E$ and then, for each curve $\alpha\in E$, we glue a cylinder of modulus $t_\alpha$ to $S\setminus E$ using the parametrization $\alpha$ on each boundary circle.

This construction recovers some well-studied paths in Teichm\"uller space. If each component of $E$ is geodesic for the hyperbolic metric on $S$, then for every $t$ the set $\{\, S_{\lambda t} \mid \lambda \geq 0 \, \}$ is called a grafting ray \cite{Dumas}. By taking the multicurve $E$ to be the set of core curves for a Jenkins-Strebel quadratic differential, we obtain a Strebel ray, which is a special kind of geodesic for the the Teichm\"uller metric \cite{Masur}. 

From the point of view of hyperbolic geometry, the effect of conformal grafting is to pinch the surface $S$ along the multicurve $E$. Let $S_\infty$ be the surface constructed by gluing a pair of half-infinite cylinders to $S\setminus E$ for each $\alpha\in E$. In the hyperbolic metric, these half-infinite cylinders are cusp neighborhoods. Thus we've replaced each curve in $E$ by a pair of cusps. For grafting rays \cite{Hensel} and Strebel rays \cite{Masur} it is known that $S_t$ looks more and more like $S_\infty$ in the hyperbolic metric at $t$ increases. We prove this in general.

\begin{prop} \label{geomconv}
The surface $S_t$ converges geometrically to $S_\infty$ as $t \to \infty^E$.
\end{prop}

The notation $t \to \infty^E$ means that $t_\alpha \to \infty$ for each $\alpha \in E$. There are several equivalent ways to describe geometric convergence. The first one is to say that for every choice of basepoint $x_\infty \in S_\infty$, there is a basepoint $x_t \in S_t$ such that $(S_t, x_t) \to (S_\infty, x_\infty)$ in Gromov's bilipschitz metric. This means that someone with blurred vision and limited eyesight cannot distinguish $S_t$ from $S_\infty$ when standing at $x_t$ and $x_\infty$ whenever $t$ is large enough. The second way is to say that the deck group for the universal cover of $S_t$ converges to the deck group for the universal cover of $S_\infty$ in the Chabauty topology. Lastly, we can say that for any choice of Fenchel-Nielsen coordinates for $S$ compatible with the multicurve $E$, the coordinates for $S_t$ about curves not in $E$ converge to the corresponding coordinates for $S_\infty$, and the length of every curve in $E$ converges to zero, corresponding to the fact that the curve has become a pair of cusps in $S_\infty$. We prove these three versions of Proposition \ref{geomconv} in section \ref{sectiongc}.

Our main result, however, is that the remaining Fenchel-Nielsen coordinates --the twist coordinates about the curves which are getting pinched-- converge as well.

\begin{thm}\label{twists converge}
For every $\alpha\in E$, the Fenchel-Nielsen twist coordinate for $S_t$ around $\alpha$ converges to some finite value as $t \to \infty^E$.
\end{thm}

This was first proved by Chris Judge (unpublished) in the case of Strebel rays. Our proof is both more general and elementary. One may interpret this result as saying that $S_t$ does not spiral in moduli space as it converges to $S_\infty$. Scott Wolpert asked in \cite{Wolpert} whether Fenchel-Nielsen twist coordinates about curves that get pinched stay bounded along Weil-Petersson geodesics of finite length. This would follow if one could prove that any such geodesic is contained in a compactly generated family of grafting hyperoctants of the form $t \mapsto S_t$. 

In the next section we briefly review Fenchel-Nielsen coordinates. Then we state a more precise version of Theorem \ref{twists converge} and proceed with the proof.

\section{Fenchel-Nielsen coordinates}

Recall that $E$ is an analytic multicurve in the hyperbolic Riemann surface $S$. Let $F$ be a maximal multicurve in $S$ containing $E$, and let $X$ be an element of $\teich(S)$, the Teichm\"uller space of $S$. Every $\alpha\in F$ is freely homotopic to a unique simple closed geodesic $\alpha^*$ in the hyperbolic metric on $X$. The length of $\alpha^*$ is denoted by $\ell^\alpha(X)$. If we cut $X$ along the set of closed geodesics $\bigcup_{\alpha \in F} \alpha^*$, we obtain a disjoint union of hyperbolic pairs of pants. Each such pair of pants has three geodesic seams, one between each pair of ends (either cusp neighborhoods or one-sided neighborhoods of geodesics). For each $\alpha \in F$, the half-twist coordinate $\theta^\alpha(X)$ is defined as the signed distance between the seams of the pants on both sides of $\alpha^*$, counted in number of full rotations. Since there are two diametrically opposed seams to choose from on each side of $\alpha^*$, this half-twist is only defined up to half integers, hence the name. We thus consider $\theta^\alpha(X)$ as an element of the circle $\R / \frac{1}{2} \Z$. One can lift $\theta^\alpha$ to a continuous map $\widetilde \theta^\alpha: \teich(S) \to \R$. The resulting map
$$
\begin{array}{ccc}
\teich(S) & \to & (\R_+ \times \R)^F\\
X  &  \mapsto   & (\ell^\alpha(X),\widetilde \theta^\alpha(X))_{\alpha \in F}
\end{array}
$$
is a homeomorphism, known as \emph{Fenchel-Nielsen coordinates}.

We will show that for every $\alpha \in E$, the half-twist $\theta^\alpha(S_t)$ converges as $t \to \infty^E$. Since the map $t \mapsto S_t$ is continuous and the projection $\R \to \R / \frac{1}{2} \Z$ is a covering map, the Fenchel-Nielsen twist coordinate $\widetilde \theta^\alpha(S_t)$ also converges as $t \to \infty^E$.

\section{Limiting angles} \label{limiting angles}

In this section, we explain how the limit of each half-twist $\theta^\alpha(S_t)$ for $\alpha \in E$ can be seen on the limit surface $S_\infty$.

If we add two copies of $E$ to the cut-up surface $S \setminus E$, we get a bordered Riemann surface $S_E$. Let $\Sp^1=\R / \Z$ have the usual orientation. For each $\alpha \in E$, there are two parametrized boundary curves $$\alpha^+,\alpha^- : \Sp^1 \to S_E,$$ labeled in such a way that $S_E$ is to the left of $\alpha^+$ and to the right of $\alpha^-$. The surface $S_\infty$ is obtained by gluing half-infinite cylinders
\begin{eqnarray*}
\Sp^1 \times (-\infty,0] &\mbox{and}& \Sp^1 \times [0,+\infty)
\end{eqnarray*}
to $S_E$ along $\Sp^1 \times \{0\}$ via the maps $\alpha^+$ and $\alpha^-$ respectively, for each $\alpha \in E$. Note that a half-infinite cylinder is conformally equivalent to a punctured disk. It is easy to see that every component of $S_\infty$ is a hyperbolic surface of finite area.

\begin{figure}[htbp]
\centering
\includegraphics[scale=.8]{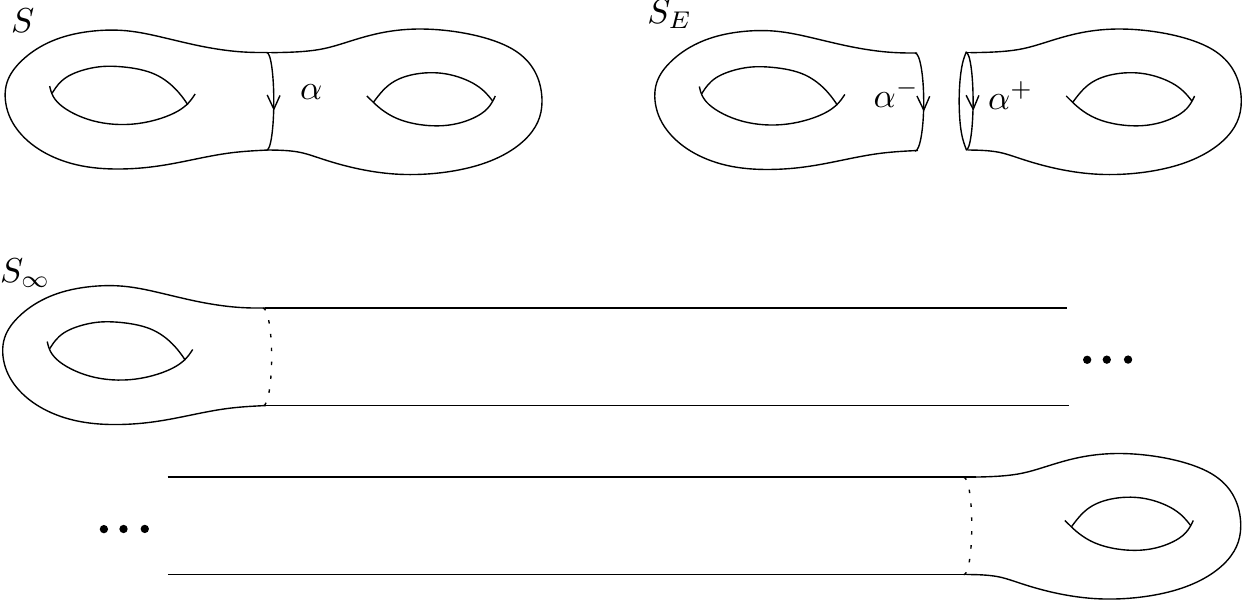}
\caption{The infinitely grafted surface $S_\infty$.}
\end{figure}

As before, let $F$ be any maximal multicurve in $S$ containing $E$. The inclusion of $F \setminus E$ in $S_\infty$ is then a maximal multicurve. For each $\beta \in F \setminus E$, there is a unique closed geodesic $\beta^*$ freely homotopic to $\beta$ in $S_\infty$. The geodesic multicurve
$$
(F \setminus E)^* = \bigcup_{\beta \in F \setminus E} \beta^*
$$
divides $S_\infty$ into hyperbolic pairs of pants, each having three geodesic seams.

For the rest of this section, fix some $\alpha \in E$. Denote by $C^+=\Sp^1 \times (-\infty,0]$ and $C^-=\Sp^1 \times [0,+\infty)$ the half-infinite cylinders in $S_\infty$ glued to $S_E$ via $\alpha^+$ and $\alpha^-$ respectively. For each of the two cusp neighborhoods $C^+$ and $C^-$, there are two seams for the pants decomposition $(F \setminus E)^*$ that go out towards this cusp. Pick one seam $\eta^+$ going to infinity in $C^+$ and one seam $\eta^-$ going to infinity in $C^-$. Let $(\theta^\pm(s),y^\pm(s))$ be the coordinates of $\eta^\pm(s)$ in $C^\pm$.

\begin{lem} \label{lem:asymangle}
There exist angles $\theta^\pm_\infty \in \Sp^1$ such that $\theta^\pm(s) \to \theta_\infty^\pm$ as $s\to\infty$.
\end{lem}
\begin{proof}
Consider the conformal isomorphism $\varphi : \overline \D \setminus \{ 0 \} \to C^+$ as a map into $S_\infty$. There is a covering map
$
\psi : \D \setminus \{ 0 \} \to S_\infty 
$
corresponding to the cusp $C^+$, and $\varphi$ lifts under $\psi$ to a conformal embedding $\widetilde \varphi : \overline \D \setminus \{ 0 \} \to \D \setminus \{ 0 \}$. We can also lift the seam $\eta^+$ under the covering map $\psi$. The resulting geodesic $\widetilde{\eta^+} : [0,\infty) \to \D \setminus \{ 0 \}$ goes towards the puncture as $s \to \infty$.  Therefore, it is a radial ray.

By Riemann's removable singularity theorem,  $\widetilde \varphi$ extends holomorphically at the origin with $\widetilde \varphi(0)=0$. In particular, the path
$$
\varphi^{-1} \circ \eta^+ =  \widetilde \varphi^{-1} \circ \widetilde{\eta^+}
$$
has a one-sided tangent at the origin. Equivalently, the argument of $\varphi^{-1} \circ \eta^+(s)$ has a limit as $s \to \infty$.
\end{proof}

The precise version of Theorem \ref{twists converge} which we will prove is the following.

\begin{thm} \label{precise twists}
The Fenchel-Nielsen half-twist $\theta^\alpha(S_t)$ converges to $\theta_\infty^+ - \theta_\infty^-$ in $\R/\frac{1}{2} \Z$ as $t\to\infty^E$.
\end{thm}

\begin{figure}[htbp]
\centering
\includegraphics[scale=.8]{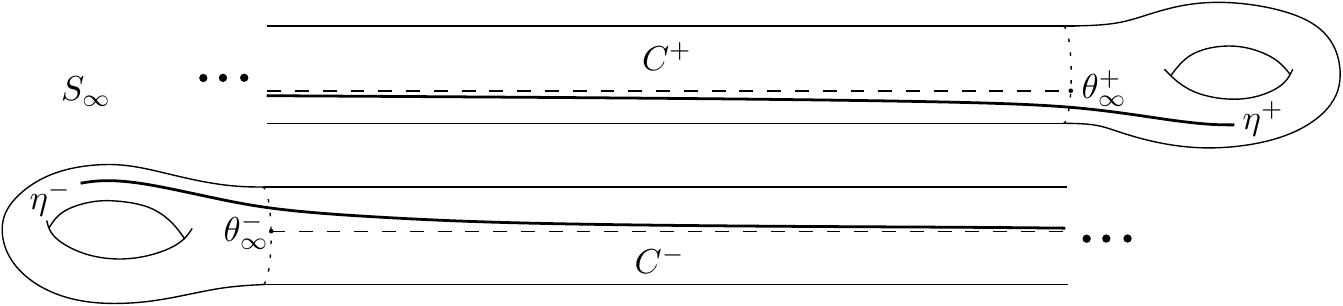}
\caption{The asymptotic angles on either side of the pair of cusps corresponding to $\alpha$ in $S_\infty$.}
\end{figure}

By definition, the half-twist $\theta^\alpha(S_t)$ is the angle difference between two seams $\eta_t^+$ and $\eta_t^-$ on either side of the closed geodesic $\alpha^*$ in $S_t$. There are two main ideas involved in proving that $\theta^\alpha(S_t) \to (\theta_\infty^+ - \theta_\infty^-)$. The first one is of geometric convergence. In the hyperbolic metric, $S_t$ looks more and more like $S_\infty$ away from the curves which are getting pinched. In particular, $\eta_t^+$ and $\eta_t^-$ converge on compact sets to $\eta^+$ and $\eta^-$. The second idea is to control what happens deep inside the grafted cylinder, and how the change from grafted cylinder coordinates to hyperbolic coordinates distorts distances.

\begin{figure}[htbp]
\centering
\includegraphics[scale=.8]{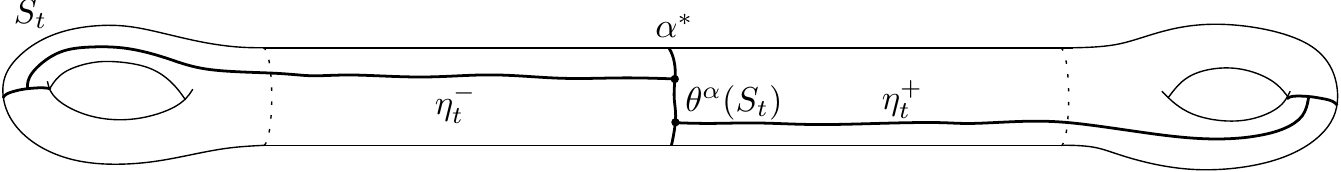}
\caption{The closed geodesic $\alpha^*$ in $S_t$ and one geodesic seam on each side.}
\end{figure}

\section{Grafted cylinders are long in the hyperbolic metric} \label{collarsingraft}

In this section, we relate the conformal geometry of $S_t$ to its hyperbolic geometry. Let us set up some notations. For every $t\in (\R_{\geq 0})^E$, the surface $S_t$ is constructed by gluing, for each $\alpha \in E$, the cylinder $\Sp^1 \times [0,t_\alpha]$ to the bordered surface $S_E$ with the maps $\alpha^+$ and $\alpha^-$ on the top and bottom circles respectively. Let $C^\alpha_t$ be the grafted cylinder $\Sp^1 \times (0,t_\alpha)$ in $S_t$, and let $\alpha_t$ be the central circle $\Sp^1 \times \{t_\alpha /2 \}$. Our goal is to prove that the hyperbolic distance across either half of $C^\alpha_t$ goes to infinity.

\begin{lem} \label{distanceacross}
For every $\alpha \in E$, the distance across either component of $C^\alpha_t \setminus \alpha_t$ in the hyperbolic metric on $S_t$ goes to $\infty$ as $t \to \infty^E$.
\end{lem}

This will be used in proving that $S_t$ converges geometrically to $S_\infty$ as $t \to \infty^E$.

\subsection*{Grafting pinches}

As we mentioned in the introduction, grafting $S$ along $E$ pinches the corresponding geodesics. Recall that $\alpha_t^*$ denotes the closed geodesic homotopic to $\alpha_t$ in $S_t$, and that $\ell^\alpha(S_t)$ denotes its length.

\begin{lem} \label{short geodesics}
For every $\alpha \in E$ and every $t \in (\R_{\geq 0})^E$, we have $\pi / \ell^{\alpha}(S_t)\geq t_\alpha.$
\end{lem}
\begin{proof}
For $\ell>0$, let $\ell \Sp^1 = \R / \ell \Z$. The hyperbolic metric on $\ell\Sp^1 \times ( -\frac{\pi}{2} , \frac{\pi}{2})$ is given by $ds = \sqrt{dx^2 + dy^2}/\cos y.$ There is a unique simple closed geodesic in this metric, namely the core curve $\ell\Sp^1 \times \{ 0 \}$, which has length $\ell$.

Assume that $t_\alpha > 0$. By the Schwarz lemma, the inclusion $C_t^\alpha \hookrightarrow S_t$ is a contraction with respect to the hyperbolic metrics. Since $C_t^\alpha$ is conformally equivalent to the cylinder $\frac{\pi}{t_\alpha} \Sp^1 \times (-\frac{\pi}{2} , \frac{\pi}{2})$, we have
\begin{equation*}
\pi/t_\alpha = \length(\alpha_t, C_t^\alpha) \geq \length(\alpha_t, S_t) \geq \length(\alpha_t^*, S_t)= \ell^\alpha(S_t). \qedhere
\end{equation*}
\end{proof}

\subsection*{The modulus of a cylinder}

A \textit{conformal metric} on a Riemann surface $Z$ is a Borel measurable function
$$
\rho : TZ \to \R_{\geq 0}
$$
such that $\rho(\lambda v) = |\lambda| \rho(v)$ for every $\lambda \in \Co$ and every $v \in TZ$. Let $\Gamma$ be any family of curves in $Z$. We say that a conformal metric $\rho$ is \textit{admissible for $\Gamma$} if the length
$$
\length(\gamma,\rho) := \int \rho(\gamma'(t)) \,dt
$$
is at least $1$ for every locally rectifiable $\gamma \in \Gamma$. The \textit{modulus} of $\Gamma$ is defined as
$$
\Mod \Gamma := \inf \left\{\, \int_Z \rho^2 : \text{$\rho$ is admissible for $\Gamma$} \,\right\}.
$$
This is the same as the reciprocal of the extremal length of $\Gamma$ and is a conformal invariant.

Any Riemann surface which is homotopy equivalent to a circle will be called a \textit{cylinder} or an \textit{annulus}. The modulus of a cylinder is by definition the modulus of the family of its essential loops. The cylinders $\Sp^1 \times \R$ and $\Sp^1 \times (0,\infty)$ have infinite modulus, and $\Sp^1 \times (0,m)$ has modulus $m$. Every cylinder is conformally equivalent to exactly one of these model cylinders.

\subsection*{Collars in grafted cylinders}

The collar lemma states that if $\gamma$ is a simple closed geodesic of length $\ell$ in a hyperbolic surface $Z$, the neighborhood $N$ of width 
$$
w=\sinh^{-1}(1 / \sinh(\ell /2) )
$$
about $\gamma$ in $Z$ is an embedded annulus. Assume that $\ell \in (0, \pi)$ and let $m$ be the modulus of $N$. Since the covering space of $Z$ corresponding to $\gamma$ is a cylinder of modulus $\pi / \ell$, we have $m \leq \pi / \ell$. On the other hand, we have
$$
\csc(\pi/2 - m\ell /2)=\sec(m \ell / 2) = \cosh w = \coth(\ell /2)\geq \csc(\ell /2)
$$
so that $m \geq \pi / \ell - 1.$ The details are left to the reader since a different proof of this lower bound can be found in \cite{DouadyHubbard}.

The next lemma shows that for $t$ large enough, there is a hyperbolic collar of large modulus about $\alpha_t^*$ which is not only embedded in $S_t$ but also contained in the grafted cylinder $C_t^\alpha$.

We need some more definitions first. For $s, t \in (\R_{\geq 0})^E$, let us say that $s \geq t$ if $s_\alpha \geq t_\alpha$ for every $\alpha \in E$. For every $t \in (\R_{\geq 0})^E$, let $X_t=S_t \setminus \bigcup_{\alpha \in E} \alpha_t$. There is a unique conformal embedding $g_t : X_t \to S_\infty$  such that the diagram
$$
\begin{tikzcd}[column sep=tiny]
X_t \arrow{rr}{g_t} & & S_\infty  \\
& S\setminus E \arrow{ul} \arrow{ur} &
\end{tikzcd}
$$
commutes, where the inclusions $S \setminus E \hookrightarrow X_t$ and $S \setminus E \hookrightarrow S_\infty$ come from the construction of $S_t$ and $S_\infty$. Observe that $g_t(X_t)$ exhausts $S_\infty$ as $t \to \infty^E$.


\begin{lem} \label{outside collar}
There exists a constant $c>1$ and a vector $T \in (\R_{\geq 0})^E$ such that for all $t\geq T$ and all $\alpha \in E$, the hyperbolic collar of modulus $\pi/\ell^\alpha(S_t)- c$ about $\alpha_t^*$ is contained in the grafted cylinder $C_t^\alpha$.
\end{lem}
\begin{proof}
Fix $\alpha \in E$ and let $\ell_t=\ell^\alpha(S_t)$ and $C_t=C_t^\alpha$. Assume that $t_\alpha > 1$ so that $\ell_t < \pi$ by Lemma \ref{short geodesics}.

Let $\psi_t : \ell_t \Sp^1 \times (-\frac{\pi}{2},\frac{\pi}{2}) \to S_t$ be the annulus cover corresponding to $\alpha_t$. Consider the largest subset $J_t \subset (-\frac{\pi}{2},\frac{\pi}{2})$ such that
$\psi_t (\ell_t\Sp^1 \times J_t) \subset C_t$. To prove the lemma, we have to find $c>0$ such that $(-\frac{\pi-c \ell_t}{2} , \frac{\pi-c \ell_t}{2})$ is contained in $J_t$ for $t$ large enough.

It is easy to see that $J_t$ is open. Moreover, $J_t$ is connected. Indeed, if $u,v \in J_t$ and $u<v$, then the image curves $\psi_t (\ell_t \Sp^1 \times \{u\})$ and $\psi_t (\ell_t \Sp^1 \times \{v\})$ are homotopic to each other inside $C_t$, as they are both homotopic to $\alpha_t$. Such a homotopy lifts under $\psi_t$ to a homotopy between $\ell_t \Sp^1 \times \{u\}$ and $\ell_t \Sp^1 \times \{v\}$ inside $\psi^{-1}(C_t)$. This homotopy has to sweep all points in $\ell_t \Sp^1 \times [u,v]$. Therefore $\psi_t (\ell_t \Sp^1 \times [u,v]) \subset C_t$ and $[u,v] \subset J_t$. It follows that $J_t$ is an open interval.

Let $I_t = (-\frac{\pi}{2},\frac{\pi}{2}) \setminus J_t$. If we find $c>0$ such that $|I_t| \leq c \ell_t/2$, then the result follows. If $I_t'= I_t \cap (-\frac{\pi-\ell_t}{2},\frac{\pi-\ell_t}{2})$, then $|I_t|\leq |I_t'| + \ell_t$ so 
a bound of the form $|I_t'| \leq M \ell_t$ suffices, for we can then set $c:=2(M+1)$. The problem is thus reduced to finding an upper bound $M$ for $$\Mod(\ell_t \Sp^1 \times I_t') = |I_t'|/\ell_t.$$

Let $\Gamma_t'$ denote the family of circles $\ell_t \Sp^1 \times \{ u \}$ such that $u\in I_t'$. We have
$$
\Mod (\ell_t \Sp^1 \times I_t') = \Mod \Gamma_t' = \Mod \psi_t(\Gamma_t').
$$
The last equality holds because the restriction of $\psi_t$ to $\ell_t \Sp^1 \times (-\frac{\pi-\ell_t}{2},\frac{\pi-\ell_t}{2})$ is injective by the lower bound on the modulus of collars. This is why we replaced $I_t$ with $I_t'$.

Now let $\Gamma_t$ be the family of closed curves in ${S}_t$ which are homotopic to $\alpha_t$ but are not entirely contained in the grafted cylinder $C_t$. Observe that $\psi_t(\Gamma_t') \subset \Gamma_t$ by construction. Monotonicity of modulus thus implies that $\Mod \psi_t(\Gamma_t') \leq \Mod \Gamma_t.$

Let $\rho$ be the hyperbolic metric on $S_\infty$ and let $A$ be its area $\int_{S_\infty} \rho^2$ which is finite. Define the conformal metric $\rho_t$ on $S_t$ to be $g_t^* \rho$ on $X_t=S_t \setminus \bigcup_{\alpha \in E} \alpha_t$. The area of $\rho_t$ is at most $A$. Therefore, it suffices to find a lower bound $L>0$ for the length of the curves in $\Gamma_t$ in the metric $\rho_t$, for then rescaling the metric by $1/L$ will give an upper bound of $A/L^2$ for $\Mod \Gamma_t$. 

For every $\beta \in E$, there are two corresponding grafted half-infinite cylinders $C^\beta_+$ and $C^\beta_-$ in $S_\infty$. Let $\psi_\pm^\beta : \Sp^1 \times (0,\infty) \to S_\infty$ be the corresponding covering maps with sections $s_\pm^\beta : C_\pm^\beta \to \Sp^1 \times (0,\infty)$. By Riemann's removable singularity theorem, there exists a $u>0$ such that each image $s_\pm^\beta(C_\pm^\beta)$ contains a half-infinite cylinder of the form $\Sp^1 \times [u,\infty)$. Thus $\psi_{\pm}^\beta(\Sp^1 \times [u,\infty))$ is contained in $C_\pm^\beta$ for every $\beta \in E$.    

Let $L$ be twice the injectivity radius of $\Sp^1 \times (0,\infty)$ at the point $iu$ in the hyperbolic metric. Then every essential loop in $\Sp^1 \times (0,\infty)$ which is not entirely contained in $\Sp^1 \times [u,\infty)$ has length at least $L$. 

If $\gamma \in \Gamma_t$ is contained in $Y_t$, then $g_t(\gamma)$ is not contained in either $C_+^\alpha$ or $C_-^\alpha$ but is homotopic to one of them. Therefore $g_t(\gamma)$ lifts under the corresponding covering map $\psi_\pm^\alpha$ to an essential curve in $\Sp^1 \times (0,\infty)$ not entirely contained in $\Sp^1 \times [u,\infty)$. It follows that 
$$
\length(\gamma, \rho_t) = \length(g_t(\gamma),\rho)  \geq L.
$$

If $\gamma \in \Gamma_t$ is not contained in $X_t$, then a subarc $\omega$ of $\gamma$ has to cross half of a grafted cylinder, that is one component of $C_t^\beta \setminus \beta_t$ for some $\beta \in E$. If $t$ is large enough, then any such arc is longer than $L$ in the metric $\rho_t$. To see this, consider the union of cylinders $$K=\bigcup_{\beta \in E} \left( \psi_+^\beta(\Sp^1 \times [u,e^L u])\cup \psi_-^\beta(\Sp^1 \times[u,e^L u]) \right)$$
contained in $\bigcup_{\beta \in E} C^\beta_+ \cup C^\beta_-$. As $K$ is compact, it is contained in $g_t(X_t)$ for all large enough $t$. Then every arc crossing a component of $C_t^\beta \setminus \beta_t$ has to cross a component of $g_t^{-1}(K)$ so that
\begin{equation*}
\length(\gamma, \rho_t) \geq \length(\omega, \rho_t) = \length(g_t(\omega),\rho)  \geq L. \qedhere
\end{equation*}
\end{proof}

\subsection*{The non-squeezing lemma}

Any cylinder of large modulus, when embedded in the plane, must contain a large round annulus \cite[\S 4.11]{Ahlfors} \cite[\S 2.1]{McMullen}. We give a new elementary proof.

\begin{lem} \label{straightsubcylinder}
Every essential cylinder in $\Co / \Z$ of modulus $m>1$ contains a closed straight cylinder of modulus $(m-1)$.
\end{lem}
\begin{proof}
Let $b>0$ and let $X$ be an essential cylinder in $\Co / \Z$ which contains no closed straight cylinder of modulus $b$. We want to show that $\Mod X < b+1$. 

For convenience of notation, we transfer to the planar setting using the conformal map $h : \Co / \Z  \to  \Co\setminus \{0\}$ given by $h(z)=e^{2\pi i z}$. The complement of $Y=h(X)$ in the Riemann sphere has two connected components. Let $M$ be the component containing $0$ and $N$ the one containing $\infty$. For $I\subset [0,\infty]$, let us write $AI= \{\, z \in \CHAT : |z| \in I \,\}.$ There is a smallest $r\in(0,\infty)$ such that $M \subset A[0,r]$ and a largest $R\in(0,\infty)$ such that $N \subset A[R,\infty]$. Note that no essential loop in $Y$ is entirely contained in $A[0,r]$ or $A[R,\infty]$, for otherwise $M$ or $N$ would be disconnected. It may happen that $r>R$ if $M$ and $N$ intertwine. If this is the case, let $M'$ be the connected component of $M\cap A[0,R]$ containing the origin and consider the cylinder $Y':=\widehat{\Co} \setminus (M'\cup N)$. Remark that $h^{-1}(Y')$ does not contain any straight cylinder. Also, $Y \subset Y'$ and the inclusion is essential, so that $\Mod Y \leq \Mod Y'$ by monotonicity of modulus. Since our goal is to bound $\Mod Y$ from above, we can replace $Y$ with $Y'$. In other words, we can assume that $r\leq R$. 

We now go back to cylindrical coordinates. Let $u= \log r/2\pi$ and $v= \log R/2\pi$. Then every essential loop in $X$ intersects $S^1 \times (u,\infty)$ and $S^1 \times (-\infty,v)$ by a previous observation. Moreover, $S^1 \times (u,v)$ is contained in $X$ since $A(r,R)\subset Y$. By assumption, $0\leq v-u \leq b$.

Consider the cylinder $Z=S^1 \times (u-1/2, v+1/2)$. We define the conformal metric $\rho$ to be the euclidean metric $|dz|$ on $X\cap Z$ and zero on $X \setminus Z$. Let $\gamma$ be an essential loop in $X$. We claim that $\length(\gamma,\rho) \geq 1$. If $\gamma$ is contained in $Z$, this is obvious. If not, then $\gamma$ intersects $S^1 \times (-\infty,u-1/2]$ or $S^1 \times [v+1/2,\infty)$. Without loss of generality, assume that $\gamma$ intersects the former. Since $\gamma$ has to intersect $S^1 \times (u,\infty)$ as well, its length is at least $1$, for two disjoint subarcs must cross $S^1 \times(u-1/2,u]$.

Therefore, $\rho$ is admissible for the family of essential loops in $X$ and we have
\begin{equation*}
\Mod X \leq \int_X \rho^2 \leq v-u+1 \leq b+1.
\end{equation*}
It is easy to see that $\rho$ is not extremal and hence the strict inequality holds.

\end{proof}

\begin{remark}
Among essential cylinders in $\Co\setminus \{0\}$ which do not contain any closed round annulus of modulus $b$ centered at the origin, the one with largest modulus is the Teichm\"uller annulus $\Co \setminus [-1,0] \cup [e^{2\pi b} , \infty)$ The latter has modulus between $b+\frac{4\log 2}{2\pi}$ and $b+\frac{5\log 2}{2\pi}$ \cite[\S 4.12]{Ahlfors}.
\end{remark}

\subsection*{Getting across grafted cylinders}

We are ready to prove that the hyperbolic distance across either half of the grafted cylinder $C_t^\alpha$ goes to infinity as $t \to \infty^E$.

\begin{proof}[Proof of Lemma \ref{distanceacross}]
Let $\ell_t=\ell^\alpha(S_t)$, $C_t=C_t^\alpha$ and let $\psi_t : \ell_t\Sp^1 \times (-\frac{\pi}{2},\frac{\pi}{2}) \to S_t$ be annulus cover corresponding to $\alpha_t$. Also let $c$ and $T$ be as in Lemma \ref{outside collar}.

Let $C_t^\pm$ denote the two components of $C_t \setminus \alpha_t$. Each of $C_t^\pm$ has modulus $t_\alpha / 2$. The cylinders $C_t^\pm$ lift to the covering annulus $\ell_t\Sp^1 \times (-\frac{\pi}{2},\frac{\pi}{2})$. By Lemma \ref{straightsubcylinder}, those lifts contain straight cylinders of of modulus $b_t=t_\alpha / 2 - 1$ and hence euclidean height $\ell_t b_t$. Therefore any path crossing one of $C_t^\pm$ lifts to a path in $\ell_t\Sp^1 \times (-\frac{\pi}{2},\frac{\pi}{2})$ with endpoints separated by a straight cylinder of euclidean height $\ell_t b_t$. By Lemma \ref{outside collar}, if $t\geq T$ then one endpoint is also outside the collar $\ell_t\Sp^1 \times (-\frac{\pi - c \ell_t}{2},\frac{\pi - c \ell_t}{2})$. Assume that $t_\alpha$ is large enough so that $\pi-c\ell_t > 0$. The length of any crossing path is then at least
$$
s=\int_u^v \frac{1}{\cos y} \mathrm{d}y,
$$
where 
$$u = \frac{\pi-\ell_t (c+2b_t)}{2} \mbox{ and } v =  \frac{\pi-\ell_t c }{2}. $$ 

Since $1/\cos y \geq 2 / (\pi - 2|y|)$ for $y \in (-\frac{\pi}{2},\frac{\pi}{2})$, we have $$ s \geq \int_{u}^{v}\frac{2 \mathrm{d}y}{\pi - 2y} = \int_{c }^{c+2b_t}\frac{\mathrm{d}y}{y} = \log(1+2b_t/c).$$ This lower bound goes to infinity as $t \to \infty^E$.
\end{proof}

\section{Convergence of surfaces} \label{sectiongc}

The goal of this section is to prove that $S_t$ converges geometrically to $S_\infty$ as $t \to \infty^E$, which is the content of Proposition \ref{geomconv}.

\subsection*{Gromov convergence}

Fix some point $x\in S \setminus E$ and some nonzero tangent vector $v$ at $x$. For every $t \in (\R_{\geq 0})^E \cup \{\infty\}$, denote by $x_t$ and $v_t$ the images of $x$ and $v$ under the inclusion $S\setminus E \hookrightarrow S_t$. 

For every $t \in (\R_{\geq 0})^E \cup \{\infty\}$, let $p_t : \D \to S_t$ be the unique holomorphic covering map such that $p_t(0)=x_t$ and $(\mathrm{d}_0 p_t)(\partial / \partial z)=\lambda v_t$ for some $\lambda >0$. For every $t \in (\R_{\geq 0})^E$, let $\widetilde X_t$ be the connected component of $p_t^{-1}(X_t)$ containing the origin. We want to lift the conformal embedding $g_t: X_t \to S_\infty$ to universal covers. This requires the following topological lemma.

\begin{lem} \label{ess subsurface}
The following statements hold for every $t \in (\R_{\geq 0})^E$ :
\begin{enumerate}
\item $\widetilde X_t$ is simply connected;
\item the inclusion $i: X_t \to S_t$ induces an injective homomorphism $$i_* :\pi_1(X_t,x_t) \to \pi_1(S_t,x_t);$$
\item the embedding $g_t:X_t \to S_\infty$ induces an isomorphism $$(g_t)_*: \pi_1(X_t,x_t) \to \pi_1(S_\infty,x_\infty).$$
\end{enumerate} 
\end{lem}
\begin{proof}

Let $h:S_t \to S_t$ be a homeorphism that maps each $\alpha_t$ to its geodesic representative $\alpha_t^*$ and lift $h$ to a homeomorphism $\widetilde h : \D \to \D$. 
 
The set $\widetilde{h}(\widetilde X_t)$ is a connected component of $\D \setminus p_t^{-1}(\bigcup_{\alpha \in E} \alpha_t^*)$ and is therefore hyperbolically convex. Since $\widetilde{h}$ is a homeomorphism, $\widetilde X_t$ is contractible and in particular simply connected.

The kernel of $i_* :\pi_1(X_t,x_t) \to \pi_1(S_t,x_t)$ is equal to the image under $(p_t)_*$ of $\pi_1(\widetilde X_t,0)$, which is trivial by part (1). Hence $i_*$ is injective. 

The sets $g_t(X_t)$ and $S_\infty$ deformation retract onto a common closed subset. Such a deformation retraction is obtained by contracting vertical lines in the grafted half-cylinders. It follows that the inclusion of $g_t(X_t)$ into $S_\infty$ induces an isomorphism on fundamental groups. Since $g_t$ is a homeomorphism onto its image, $(g_t)_*: \pi_1(X_t,x_t) \to \pi_1(S_\infty,x_\infty)$ is an isomorphism.
\
\end{proof}

Standard covering space theory implies that we can lift $g_t$.

\begin{lem}
For every $t \in (\R_{\geq 0})^E$, the embedding $g_t: X_t \to S_\infty$ lifts to an injective holomorphic map $\widetilde{g_t} : \widetilde X_t \to \D$ with $\widetilde{g_t}(0)=0$ and $\widetilde{g_t}'(0) > 0$.
\end{lem}

\begin{proof}
Since $\widetilde X_t$ is simply connected and $p_\infty$ is a holomorphic covering map, the composition $g_t\circ p_t: \widetilde X_t\to S_\infty$ lifts to a holomorphic map $\widetilde{g_t}:\widetilde X_t \to \D$ fixing the origin. The derivative $\widetilde{g_t}'(0)$ is positive because of the normalization of the covering maps and the fact that $g_t$ sends $v_t$ to $v_\infty$.

We claim that $\widetilde{g_t}(\widetilde X_t)$ is simply connected. First observe that the restriction $p_\infty : \widetilde{g_t}(\widetilde X_t) \to g_t(X_t)$ is a covering map, and hence $(p_\infty)_*$ is injective on $\pi_1(\widetilde{g_t}(\widetilde X_t),0)$. By covering space theory, the image of $\pi_1(\widetilde{g_t}(\widetilde X_t),0)$ in $\pi_1(g_t(X_t),x_\infty)$ is equal to the kernel of the homomorphism  $j_*:\pi_1(g_t(X_t),x_\infty) \to \pi_1(S_\infty,x_\infty)$ induced by the inclusion map $j:g_t(X_t) \to S_\infty$. By the last lemma, $j_*$ is a bijection. Consequently, $\pi_1(\widetilde{g_t}(\widetilde X_t),0)$ is trivial.

Since $\widetilde{g_t}(\widetilde X_t)$ is simply connected and $p_t$ is a covering map, $g_t^{-1} \circ p_\infty : \widetilde{g_t}(\widetilde X_t) \to X_t$ lifts to a map $h_t : \widetilde{g_t}(\widetilde X_t) \to \widetilde X_t$ fixing the origin. Then $h_t \circ \widetilde{g_t}$ is a lift of the identity map on $X_t$ fixing the origin and is thus the identity. This proves that $\widetilde{g_t}$ is injective.
\end{proof}

Next, we need to know that the open sets $\widetilde X_t$ exhaust the unit disk as $t \to \infty ^E$.

\begin{lem} \label{exhaustion}
Let $r_t$ and $R_t$ be the euclidean radii of the largest disks centered at the origin contained in $\widetilde X_t$ and $\widetilde{g_t}(\widetilde X_t)$ respectively. Then $r_t \to 1$, $R_t \to 1$, and hence $\widetilde{g_t}'(0)\to 1$  as $t\to \infty^E$.
\end{lem}
\begin{proof}
Let $d>0$. By Lemma \ref{distanceacross}, there is a $T \in (\R_{\geq 0})^E$ such that for every $t\geq T$, the hyperbolic distance $d(x_t,\bigcup_{\alpha \in E} \alpha_t)$ is bigger than $d$. Thus the disk $B_d(x_t)$ of radius $d$ about $x_t$ in $S_t$ is disjoint from $\bigcup_{\alpha \in E} \alpha_t$ and hence contained in $X_t$. Therefore, $B_d(0)$ is contained in $\widetilde X_t$. Since $d$ is arbitrary, $r_t \to 1$ as $t \to \infty^E$. 

Let $D$ be any closed disk inside $\D$ centered at the origin. The projection $p_\infty(D)$ is compact and thus contained in $g_t(X_t)$ for all large enough $t$. Hence $D$ is contained in $p_\infty^{-1}(g_t(X_t))$, and thus in $\widetilde{g_t}(\widetilde X_t)$, for all large enough $t$.  Therefore $R_t \to 1$ as $t\to \infty^E$.

By the Schwarz lemma, we have $R_t \leq \widetilde{g_t}'(0)\leq 1/r_t$ and hence $\widetilde{g_t}'(0)\to 1$  as $t\to \infty^E$.
\end{proof}

A normal families argument easily implies that $\widetilde{g_t}$ converges to the identity.

\begin{lem}\label{liftsconverge}
The maps $\widetilde{g_t}$ and $\widetilde{g_t}^{-1}$ converge locally uniformly to the identity map on $\D$ as $t \to \infty^E$.
\end{lem}
\begin{proof}
Let us prove that $\widetilde{g_t} \to \id$. By Montel's theorem, every subnet of the net $(\widetilde{g_t})$ admits a subnet which converges locally uniformly to some holomorphic limit $g : \D \to \D$. We must have $g(0)=0$, since $\widetilde{g_t}(0)=0$ for all $t$. By Cauchy's integral formula, derivatives also converge pointwise, and so $g'(0)=1$ by the last lemma. By the Schwarz lemma, $g$ is the identity. Therefore, the net $(\widetilde{g_t})$ converges to the identity map. The proof for $\widetilde{g_t}^{-1}$ is identical.
\end{proof}

From this theorem and Cauchy's integral formula, it follows that the derivatives $\widetilde{g_t}'$ and $(\widetilde{g_t}^{-1})'$ converge locally uniformly to $1$. Therefore, $\widetilde{g_t}^{-1}$ and hence $g_t^{-1}$ is as close as we want to being a hyperbolic isometry on compact sets.

\begin{prop}
The surface $(S_t,x_t)$ converges to $(S_\infty,x_\infty)$ in Gromov's bilipschitz metric as $t\to \infty^E$.
\end{prop}
\begin{proof}
Let $R>0$ and $K>1$. For $t$ large enough, the map $g_t^{-1}$ is defined on the ball $B_R(x_\infty)$ and its restriction to that disk has bilipschitz constant less than or equal to $K$. This is because the norm of the derivative of $g_t^{-1}$ at a point $p_\infty(z)$ with respect to hyperbolic metrics is equal to
$$
\frac{|(\widetilde{g_t}^{-1})'(z)|}{1-|\widetilde{g_t}^{-1}(z)|^2} \cdot (1-|z|^2)
$$
and the latter converges to $1$ uniformly on compact sets. Thus the Gromov distance between $(S_t,x_t)$ and $(S_\infty,x_\infty)$ is at most $\log K$.
\end{proof}

\subsection*{Convergence of deck groups}

For every $t\in(\R_{\geq 0})^E \cup \{\infty\}$, let $G_t$ be the deck group for the covering map $p_t:\D \to S_t$ and let
$$
\Theta_t : \pi_1(S_t,x_t) \to G_t
$$
be the isomorphism where $\Theta_t([\beta])$ is the unique deck transformation $h$ such that $h(0)$ is the endpoint of the lift of the loop $\beta$ based at $0$. 

Recall from Lemma \ref{ess subsurface} that the inclusion $i : X_t \to S_t$ is injective on fundamental groups and that $g_t : X_t \to S_\infty$ is bijective on fundamental groups. We thus have a faithful representation
$$
\Phi_t : \pi_1(S_\infty,x_\infty) \to \Aut(\D)
$$
defined by $\Phi_t = \Theta_t \circ i_* \circ (g_t)_*^{-1}$.

It is well-known that Gromov convergence implies Chabauty convergence of deck groups. We include a proof for completeness.

\begin{prop}\label{Chabauty convergence}
The deck group $G_t$ converges geometrically to $G_\infty$ and the representation $\Phi_t$ converges algebraically and geometrically to $\Theta_\infty$ as $t\to \infty^E$.
\end{prop}
\begin{proof}
Suppose we have a sequence $(h_{t^n})_{n\in\N}$ with $h_{t^n} \in G_{t^n}$ and $t^n \to \infty^E$ which converges locally uniformly to a M\"obius transformation $h$. Observe that $\widetilde X_{t^n}$ and $h_{t^n}(\widetilde X_{t^n})$ either coincide or are disjoint, as they are connected components of the inverse image $p_{t^n}^{-1}(X_{t^n})$. Let $U$ be a neighborhood of $h(0)$ with compact closure in $\D$. By Lemma \ref{exhaustion}, $U$ is contained in $\widetilde X_{t^n}$ for all large enough $n$. Since $h_{t^n}(0)$ converges to $h(0)$, we eventually have $h_{t^n}(0) \in \widetilde X_{t^n}$ and hence $h_{t^n}(\widetilde X_{t^n})=\widetilde X_{t^n}$. Then on $\widetilde X_{t^n}$ we have $$p_\infty \circ \widetilde{g_{t^n}} = g_{t^n} \circ p_{t^n} =  g_{t^n} \circ p_{t^n} \circ h_{t^n}=p_\infty \circ \widetilde{g_{t^n}} \circ h_{t^n}.$$
Since $\widetilde{g_{t^n}}$ converges locally uniformly to the identity map, we get $p_\infty = p_\infty \circ h$ in the limit and therefore $h$ belongs to the deck group $G_\infty$.

Let $h \in G_\infty$. We have to show that some net $(h_t)_{t\in (\R_{\geq 0})^E}$ with $h_t \in G_t$ converges locally uniformly to $h$ as $t \to \infty^E$. Assume that each component of $t$ is large enough so that $h(0) \in \widetilde{g_t}(\widetilde{X_t})$. Then $h \circ \widetilde{g_t}(\widetilde{X_t})$ equals $\widetilde{g_t}(\widetilde{X_t})$, since both are connected components of $p_{\infty}^{-1}(g_t(X_{t}))$ and their intersection is non-empty. Therefore, the map $m_t:=\widetilde{g_t}^{-1}\circ h \circ \widetilde{g_t}$ is well-defined on $\widetilde{X_t}$. We have
$$
p_t \circ m_t = p_t \circ \widetilde{g_t}^{-1}\circ h \circ \widetilde{g_t} = {g_t}^{-1}\circ p_\infty \circ h \circ \widetilde{g_t} = {g_t}^{-1}\circ p_\infty \circ \widetilde{g_t} = p_t.
$$
Since $p_t: \D \to S_t$ is a regular covering map, for every $z\in \widetilde{X_t}$ there exists a unique deck transformation $h_t^z\in G_t$ such that $h_t^z(z)=m_t(z)$. The equation $$p_t \circ h_t^z  = p_t = p_t \circ m_t$$ together with the local injectivity of $p_t$ implies that $h_t^z(w)=m_t(w)$ for all $w$ in some neighborhood of $z$. Therefore, the map $z \mapsto h_t^z$ is locally constant. As $\widetilde{X_t}$ is connected, $h_t=h_t^z$ does not depend on $z\in \widetilde{X_t}$ and we have $h_t=m_t$ on all of $\widetilde{X_t}$. Since $h_t =\widetilde{g_t}^{-1}\circ h \circ \widetilde{g_t}$ on $\widetilde{X_t}$ and $\widetilde{g_t} \to \id$, we have $h_t \to h$ as $t \to \infty^E$.

We now prove that $\Phi_t \to \Theta_\infty$ algebraically. Let $[\beta]\in \pi_1(S_\infty,x_\infty)$, let
$h_t=\Phi_t([\beta])$, and let $h = \Theta_\infty([\beta])$. Then $h(0)$ is the endpoint of the lift $\widetilde{\beta}$ of $\beta$ to $\D$ based at $0$. Suppose that all components of $t$ are large enough so that $\beta \subset g_t(X_t)$. Then by definition $h_t$ equals $\Theta_t([g_t^{-1}(\beta)])$ and hence $h_t(0)$ is the endpoint of the lift of $g_t^{-1}(\beta)$ to $\D$ based at 0. This lift is equal to $\widetilde{g_t}^{-1}(\widetilde{\beta})$,  so $$h_t(0)=\widetilde{g_t}^{-1}(h(0))=\widetilde{g_t}^{-1}\circ h \circ \widetilde{g_t}(0).$$ 
As in the previous paragraph, it follows that $h_t=\widetilde{g_t}^{-1}\circ h \circ \widetilde{g_t}$ on $\widetilde{X_t}$ and thus $h_t\to h$ as $t \to \infty^E$.

Since $\Phi_t(\pi_1(S_\infty,x_\infty))$ is a subgroup of $G_t$, and the latter converges to $G_\infty$, the only possible limits of sequences $(h_{t^n})_{n\in \N}$ with $h_{t^n} \in \Phi_{t^n}(\pi_1(S_\infty,x_\infty))$ and $t^n \to \infty^E$ are contained in $G_\infty$. By the previous paragraph, every element in $G_\infty$ arises as a limit. Therefore $\Phi_t(\pi_1(S_\infty,x_\infty))$ converges to $ \Theta_\infty(\pi_1(S_\infty,x_\infty))$ geometrically.
\end{proof}

\subsection*{Convergence of Fenchel-Nielsen coordinates in the thick part}

From Cha\-bauty convergence of the deck groups, we can easily deduce that the Fenchel-Nielsen coordinates for $S_t$ about curves that do not get pinched converge to the corresponding coordinates on $S_\infty$.

\begin{prop} \label{convergenceFN}
For every $\alpha \in F\setminus E$, we have $\ell^\alpha(S_t)\to\ell^\alpha(S_\infty)$ and $\theta^\alpha(S_t) \to \theta^\alpha(S_\infty)$ as $t\to \infty^E$.
\end{prop}
\begin{proof}
Fix $\alpha \in F\setminus E$, and choose an arc $\gamma$ from the basepoint $x$ to $\beta$ in $S\setminus E$. For any $t \in (\R_{\geq }0)^E \cup \{\infty\}$, let $\sigma_t$ denote the inclusion of the loop $\sigma = \gamma * \alpha * \overline \gamma$ in $S_t$. By Proposition \ref{Chabauty convergence}, the deck transformation $h_t = \Theta_t([\sigma_t])$ converges to $h_\infty= \Theta_\infty([\sigma_\infty])$ as $t \to \infty^E$. In particular, the translation length $\ell^\alpha(S_t)$ of $h_t$ converges to the translation length $\ell^\alpha(S_\infty)$ of $h$ as $t \to \infty^E$.

For every $\beta \in F \setminus E$ and every $t \in (\R_{\geq 0})^E \cup \{ \infty\}$, let $\beta_t^*$ be the closed geodesic homotopic to $\beta$ in $S_t$. Recall that $\theta^\alpha(S_t)$ is defined as the normalized distance between the feet of two seams $\eta_t^+$ and $\eta_t^-$ modulo one half. Each seam runs from $\alpha_t^*$ to either another simple closed geodesic in $S_t$ or a cusp. Let $\sigma^\pm$ be a loop based at $x$ in $S \setminus E$ such that its inclusion $\sigma_t^\pm$ in $S_t$ is homotopic to either the geodesic or the puncture at the other end of $\eta_t^\pm$, and let $h_t^\pm = \Theta_t([\sigma_t^\pm])$.

If $h \in \Aut(\D)$ is hyperbolic, we let $\axis(h)$ be its translation axis together with endpoints. If $h$ is parabolic, we let $\axis(h)$ be the fixed point of $h$ on $\partial \D$. By Proposition \ref{Chabauty convergence}, we have $h_t \to h_\infty$, $h_t^+ \to h_\infty^+$ and $h_t^- \to h_\infty^-$ as $t \to \infty^E$. It follows that the corresponding axes also converge. Let $\nu_t^\pm$ be the orthogeodesic between $\axis(h_t)$ and $\axis(h_t^\pm)$ in $\overline{\D}$. The arc $\nu_t^\pm \cap \D$ is a lift of $\eta_t^\pm$. Since the function which to two possibly degenerate disjoint geodesics in $\overline \D$ assigns their orthogeodesic is continuous, we have $\nu_t^\pm \to \nu_\infty^\pm$. In particular, the distance $\Delta_t$ between the feet of $\nu_t^+$ and $\nu_t^-$ on $\axis(h_t)$ converges to the distance $\Delta_\infty$ between the feet of  $\nu_\infty^+$ and $\nu_\infty^-$ on $\axis(h_\infty)$. Modulo one half, we have $\theta^\alpha(S_t)= \Delta_t / \ell^\alpha(S_t)$ and $\theta^\alpha(S_\infty)= \Delta_\infty / \ell^\alpha(S_\infty)$. Since lengths converge, we have $\theta^\alpha(S_t) \to \theta^\alpha(S_\infty)$.

\end{proof}
 
The Fenchel-Nielsen twist coordinate $\widetilde \theta^\alpha(S_t)$ also converges as $t\to \infty^E$ since the quotient map $\R \to \R / \frac{1}{2}\Z$ is a covering map.

Recall that our ultimate goal is to prove that the half-twists about curves in $E$ converge. The natural coordinate system in which to measure the half-twist $\theta^\alpha(S_t)$ is the annulus cover corresponding to $\alpha_t^*$. In that annulus, the geodesic $\alpha_t^*$ is longitudinal and the orthogeodesics $\eta_t^+$ and $\eta_t^-$ are latitudinal. However, our only way to compare what happens near $\alpha_t^*$ in $S_t$ with what happens in $S_\infty$ is via grafted cylinder coordinates. Therefore, we need to understand how distances get distorted when we change coordinates from the grafted cylinder $C_t^\alpha$ to the corresponding annulus cover.  
 
\section{A distortion theorem for cylinders} \label{longcylinders}

Every conformal embedding from the infinite cylinder $\Co / \Z$ to itself is an isometry for the Euclidean metric. In this section, we prove that every essential conformal embedding from a sufficiently long cylinder into $\Co / \Z$ is as close as we wish to an isometry away from the boundary. For $b>0$, we write $B(b)=\left\{ z\in \Co : |\im z|<b \right\}$ for the strip of height $2b$ centered on the real line.

\begin{thm} \label{long cylinders cant twist}
Let $\varepsilon>0$. There exists $c>0$ such that for every $b>c$, and every conformal embedding $\varphi : B(b) / \Z \to \Co/\Z $ with $\varphi(0)=0$ and $\varphi_*(1)=1$ on fundamental groups, the restriction of $\varphi$ to $B(b-c) / \Z$ is within $\varepsilon$ of the identity.
\end{thm}

The theorem and its proof are similar to Koebe's distortion theorem \cite[\S 2.3]{Duren} \cite[\S 5.1]{Ahlfors}. In fact, one can deduce a weak version of Koebe's theorem from the above.

The first step is to make use of the non-squeezing lemma for cylinders, which plays the role of Koebe's one quarter theorem. For $b>0$, let $\mathcal{F}_b$ be the set of all conformal embeddings $\varphi: B(b)/\Z \to \Co / \Z$ with $\varphi(0)=0$ and $\varphi_*(1)=1$ on fundamental groups. 

\begin{lem} \label{height M}
Let $b>2$ and $\varphi \in \mathcal{F}_b$. For every straight cylinder $T$ of modulus $1$ contained in $B(b-2)/\Z$, the image $\varphi(T)$ is contained in a straight cylinder of modulus at most $2$.
\end{lem}
\begin{proof}
The image $\varphi(B(b)/\Z)$ is an essential subcylinder of $\Co / \Z$ of modulus $2b$. By Lemma \ref{straightsubcylinder}, $\varphi(B(b)/\Z)$ contains a straight cylinder $C$ of modulus $2b-1$, and the inverse image $\varphi^{-1}(C)$ contains a straight cylinder of modulus $2b-2$. In particular, $\varphi^{-1}(C)$ contains $B(b-2)/\Z$. Therefore $\varphi(B(b-2)/\Z)$ is contained in the straight cylinder $C$ of finite modulus which is itself contained in $\varphi(B(b)/\Z)$. 

Let $T$ be a straight cylinder of modulus $1$ contained in $B(b-2)/\Z$, and let $U$ be the smallest straight cylinder containing $\varphi(T)$. By the above discussion, the modulus $M$ of $U$ is finite and $\varphi^{-1}$ is defined on $U$. Suppose $M>2$ and let $\varepsilon:=(M-2)/2$. Then $\varphi^{-1}(U)$ is an essential cylinder in $\Co/\Z$ of modulus $M>2+\varepsilon$. The largest straight cylinder $V$ that $\varphi^{-1}(U)$ contains has modulus at least $1+\varepsilon$ by the Lemma \ref{straightsubcylinder}. As $T$ is a straight cylinder in $\varphi^{-1}(U)$, $V$ contains $T$. Since $\Mod V > \Mod T$, the cylinder $V$ forms an open neighborhood of one of the two boundary components of $T$. Now, $U$ contains $\varphi(V)$ and is thus an open neighborhood of one boundary component of $\varphi(T)$. This shows that $U$ is not minimal, a contradiction. Therefore $U$ has modulus at most $2$.  
\end{proof}

Each $\varphi \in \mathcal{F}_b$ lifts to a unique conformal embedding $f : B(b) \to \Co$ with $f(0)=0$ and $f(z+1)-f(z)=1$. Conversely, any $\Z$-equivariant univalent map $f : B(b) \to \Co$ fixing the origin descends to an element in $\mathcal{F}_b$ under the projection map $p:\Co \to \Co/\Z$. Let $\widetilde{\mathcal{F}_b}$ denote the set of such maps $f$. It will be more convenient to do calculations with $\widetilde{\mathcal{F}_b}$.

\begin{lem}
Let $b>2$, $f \in \widetilde{\mathcal{F}_b}$, and let $Q$ be an open unit square in $B(b-2)$ whose edges are parallel to the coordinate axes. Then the area of $f(Q)$ is at most $2$.
\end{lem}
\begin{proof}
Let $\varphi : B(b)/\Z \to \Co/\Z$ be the unique map such that $p \circ f = \varphi \circ p$. The image $p(Q)$ is included in a straight cylinder $T$ of modulus 1 contained in $B(b-2)/\Z$. By the previous lemma, $\varphi(T)$ is contained in a straight cylinder of modulus, and hence area, at most $2$. Therefore $\varphi(p(Q))$ has area at most $2$. We have $\varphi(p(Q))=p(f(Q)).$ Since $p(Q)$ is simply connected and $\varphi$ is an embedding, $\varphi(p(Q))$ is also simply connected. It follows that $p: f(Q) \to \varphi(p(Q))$ is injective. Since $p$ is a local isometry, the area of $f(Q)$ is bounded by $2$.
\end{proof}

For the rest of this section, suppose that $b>3$ and $f\in\widetilde{\mathcal{F}_b}$. We let $C_1 = \sqrt{8/\pi}$, $C_2 = 8\pi e^{5\pi}(C_1+1)$ and write $z=x+iy$ and $w=u+iv$ throughout.

\begin{lem} \label{derivative bounded}
The inequality $|f'(z)|\leq C_1$ holds for every $z\in B(b-5/2)$.
\end{lem}
\begin{proof}
Let $z\in B(b-5/2)$. Since $|f'|^2$ is subharmonic, we have
$$
|f'(z)|^2 \leq \frac{4}{\pi} \int_D |f'(u+iv)|^2 \mathrm{d}u\mathrm{d}v,
$$
where $D$ is the disk of radius one-half about $z$. Alternatively, this follows from Cauchy's formula. If $Q$ is the open unit square centered at $z$, we get
$$
|f'(z)|^2  \leq   \frac{4}{\pi} \int_Q |f'(u+iv)|^2 \mathrm{d}u\mathrm{d}v   =  \frac{4}{\pi} \int_{f(Q)} 1\ \mathrm{d}u\mathrm{d}v \leq  \frac{8}{\pi}. \qedhere
$$
\end{proof}

By integration we easily obtain the following bound.
\begin{cor} \label{sublinear}
The inequality $|f(z)|\leq C_1|z|$ holds for every $z\in B(b-5/2)$.
\end{cor}

The main trick appears in the next lemma. We apply Cauchy's residue theorem to a well-chosen function and get a good bound on $f''$.
\begin{lem} \label{second derivative}
The inequality
$$
|f''(z)|\leq C_2\, e^{2\pi(|y|-b)}
$$
holds for every $z\in B(b-3)$.
\end{lem}
\begin{proof}

As $f''$ is $\Z$-periodic, it suffices to prove the inequality for $|x| < 1/2$.  Fix $z=x+iy\in B(b-3)$ with $|x|< 1/2$ and let $$R=\left\{ w \in \Co : |\re w-x|< 1/2, |\im w|<b-5/2 \right\}.$$ 

The function 
$$w \mapsto \frac{f'(w)}{\sin^2(\pi(w-z))}$$
is holomorphic in $R\setminus \{ z \}$ and has residue $f''(z)/\pi^2$ at $z$.

By Cauchy's residue theorem, we have
$$
f''(z)=\frac{\pi}{2 i} \oint_{\partial R} \frac{f'(w)}{\sin^2(\pi(w-z))}\mathrm{d}w.
$$ 

Moreover, since the integrand is $\Z$-periodic, the two vertical sides of the integral cancel out and we are left with
$$
f''(z)=\frac{\pi}{2 i}   \int_{x - 1/2}^{x + 1/2} \left[ \frac{f'(u+ih)}{\sin^2(\pi(u+ih-z))} +  \frac{f'(u-ih)}{\sin^2(\pi(u-ih-z))} \right]\mathrm{d}u,
$$
where $h= b - 5/2$.

For $\zeta=s+it$ with $|t|\geq \log 2$, we have
$$
|\sin \zeta |= \left|\frac{e^{i\zeta}-e^{-i\zeta}}{2i} \right|\geq \frac{e^{|t|}-e^{-|t|}}{2},
$$
so that
$$
|\sin \zeta|^2 \geq \frac{e^{2|t|}+e^{-2|t|}-2}{4} \geq \frac{e^{2|t|}-2}{4} \geq \frac{e^{2|t|}}{8}.
$$
It follows that
$$
|\sin^2(\pi(u\pm ih-z))| = |\sin^2( \pi(u - x - i(y\mp h))| \geq \frac{e^{2\pi(h-|y|)}}{8}.
$$
Using this inequality and the bound $|f'|\leq C_1$ from Lemma \ref{derivative bounded} in the above integral yields

$$
|f''(z)|\leq \frac{8\pi C_1}{e^{2\pi(h-|y|)}} = 8 \pi e^{5\pi} C_1 e^{2\pi(|y|-b)} \leq C_2 e^{2\pi(|y|-b)}. \qedhere
$$
\end{proof}

The same trick can be applied to get a good bound on $f'$ on the real line, better than the one from Lemma \ref{derivative bounded}.

\begin{lem} \label{derivative on real line}
The inequality
$$
|f'(x)-1|\leq C_2\,(b+1)e^{-2\pi b}
$$
holds for all $x\in\R$.
\end{lem}
\begin{proof}
It suffices to prove the inequality on the interval $\left[-\frac{1}{2},\frac{1}{2}\right]$ as $f'$ is $\Z$-periodic. Fix $x\in \left[-\frac{1}{2},\frac{1}{2}\right]$ and let $R=\left\{ w \in \Co : |\re w - x|<\frac{1}{2}, |\im w|<b \right\}$. Since $f(w)-w$ has period 1, we have 
\begin{eqnarray*}
f'(x)-1 & = & \frac{\pi}{2 i} \oint_{\partial R} \frac{f(w)-w}{\sin^2(\pi(w-x))}\mathrm{d}w \\
& = & \frac{\pi}{2 i}   \int_{x - \frac{1}{2}}^{x + \frac{1}{2}} \left[\frac{f(u+ib)-(u+ib)}{\sin^2(\pi(u+ib-x))} + \frac{f(u-ib)-(u-ib)}{\sin^2(\pi(u-ib-x))} \right]\mathrm{d}u
\end{eqnarray*}
as in the previous lemma. By Corollary \ref{sublinear}, we have $$|f(u\pm ib)|\leq C_1|u\pm ib| \leq C_1 (b+1)$$ and thus
$$
|f(u\pm ib)-(u\pm ib)|\leq (C_1+1)(b+1).
$$
Combining this with the lower bound $|\sin (\pi(u\pm ib-x))|^2\geq e^{2\pi b}/8$ in the integral yields
$$
|f'(x)-1| \leq 8\pi(C_1+1)(b+1)e^{-2\pi b} \leq C_2 (b+1) e^{-2\pi b}. \qedhere
$$

\end{proof}

A line integral combined with the previous two lemmas yields the bound we need for the difference between $f$ and the identity. 

\begin{lem} \label{degree zero}
The inequality
$$
|f(z)-z|\leq C_2\,\frac{e^{2\pi|y|}+(b+1)^2}{e^{2\pi b}}
$$
holds for all $z \in B(b-3)$.
\end{lem}
\begin{proof}
By periodicity, we can again assume that $x \in \left[-\frac{1}{2},\frac{1}{2}\right]$, where $z=x+iy$. Then
\begin{eqnarray*}
f(z)-z & = & \int_0^x (f'(s)-1)\mathrm{d}s + \int_x^z (f'(\zeta)-1)\mathrm{d}\zeta \\
& = & \int_0^x (f'(s)-1)\mathrm{d}s  + \int_x^z \left(\int_x^\zeta f''(w)\mathrm{d}w+f'(x)-1\right)\mathrm{d}\zeta \\
& = & \int_0^x (f'(s)-1)\mathrm{d}s  +  iy(f'(x)-1 )+ \int_x^z \int_x^\zeta f''(w)\mathrm{d}w\mathrm{d}\zeta,
\end{eqnarray*}
where the integral paths are taken to be straight line segments.

By Lemma \ref{derivative on real line} we have 
\begin{eqnarray*}
\left|\int_0^x (f'(s)-1)\mathrm{d}s +  iy(f'(x)-1 )\right| &\leq& (|x|+|y|)\max_{s\in \R}|f'(s)-1| \\
& \leq & C_2\,(b+1)^2 e^{-2\pi b},
\end{eqnarray*}
and by Lemma \ref{second derivative} we obtain
\begin{eqnarray*}
\left|\int_x^z \int_x^\zeta f''(w)\mathrm{d}w\mathrm{d}\zeta\right| & \leq & \int_0^y \int_0^t |f''(x+iv)|\ \mathrm{d}v \mathrm{d}t  \\
 & \leq & C_2\,e^{-2\pi b} \int_0^y \int_0^{t} e^{2\pi |v|}\ \mathrm{d}v \mathrm{d}t  \\
 & \leq &  C_2\,e^{2\pi(|y|-b)}.
\end{eqnarray*}
Adding these two yields the result.
\end{proof}

The upper bound in Lemma \ref{degree zero} goes to zero as $b-|y|$ goes to infinity. Therefore, if $c$ is chosen large enough, $|f(z)-z|$ is as small as we wish where $b-|y| > c$, that is, on $B(b-c)$. Also, since the projection map $p : \Co \to \Co/\Z$ is a local isometry, it does not increase distances and the result holds for the class $\mathcal{F}_b$ as well. This completes the proof of Theorem \ref{long cylinders cant twist}.

\section{Convergence of twist coordinates in the thin part} \label{the proof}

Fix some curve $\alpha \in E$. We denote by $\eta_t^+$ and $\eta_t^-$ the seams orthogonal to $\alpha_t^*$ used to measure the half-twist $\theta^\alpha(S_t)$, and by $\eta^+$ and $\eta^-$ the corresponding seams in $S_\infty$. The geometric convergence $S_t \to S_\infty$ implies that $\eta_t^\pm \to \eta^\pm$ on compact sets.

\begin{lem} \label{orthocompact}
Let $K \subset S_\infty$ be compact. Then $g_t(\eta_t^\pm \cap g_t^{-1}(K)) \to \eta^\pm \cap K$ as $t \to \infty^E$.
\end{lem} 
\begin{proof}
Let $L \subset \D$ be a compact set such that $K \subset p_\infty(L)$. As in the proof of  Proposition \ref{convergenceFN}, there are lifts $\nu_t^\pm$ of $\eta_t^\pm$ and $\nu^\pm$ of $\eta^\pm$ such that $\nu_t^\pm \to \nu^\pm$ as $t \to \infty^E$. Moreover, by Lemma \ref{liftsconverge}, the maps $\widetilde{g_t}$ and $\widetilde{g_t}^{-1}$ converge to the identity on compact sets. Therefore, we have $\widetilde{g_t}(\nu_t^\pm \cap \widetilde{g_t}^{-1}(L)) \to \nu^\pm \cap L$ as $t \to \infty^E$. The result follows by projecting everything down to $S_\infty$
\end{proof}

We can now conclude with the proof of our main result, which is that $\theta^\alpha(S_t)$ converges to $ \theta_\infty^+ - \theta_\infty^-$ as $t \to \infty^E$.

\begin{proof}[Proof of Theorem \ref{precise twists}]
Let $\varepsilon>0$. We choose a buffer region $B^-=\Sp^1 \times [0,c]$ inside the grafted cylinder $C^-=\Sp^1 \times [0,+\infty)$ in $S^\infty$, and similarly $B^+ = \Sp^1 \times [-c,0]$ in $C^+ = \Sp^1 \times (-\infty,0]$, where $c>0$ is chosen in terms of $\varepsilon$. We chose $c$ large enough so that two things hold. First, we want that the angle coordinate $\theta^\pm(s)$ of $\eta^\pm(s)$ is within $\varepsilon$ of its limiting value $\theta_\infty^\pm$ whenever $\eta^\pm(s)$ is past $B^\pm$ into the cusp. This is possible by Lemma \ref{lem:asymangle}. Let $B_t^\pm = g_t^{-1}(B^\pm)$ be the top and bottom parts of $C_t^\alpha$ of height $c$ each. Let $\psi_t :A_t\to S_t$ be the covering annulus corresponding to the curve $\alpha_t$. There is a conformal embedding $s_t : C_t^\alpha \to A_t$ such that $\psi_t \circ s_t$ is the inclusion map $C_t^\alpha \hookrightarrow S_t$. We use the euclidean metric on $A_t$ instead of the hyperbolic metric, normalized so that $A_t$ has circumference $1$ like $C_t^\alpha$. Then we can think of $A_t$ as a subset of $\Co/\Z$. By Theorem \ref{long cylinders cant twist}, we can choose $c$ such that the restriction of $s_t$ to $C_t^\alpha \setminus B_t^+ \cup B_t^-$ is $\varepsilon$ close to an isometry. 

\begin{figure}[htbp]
\centering
\includegraphics[scale=.8]{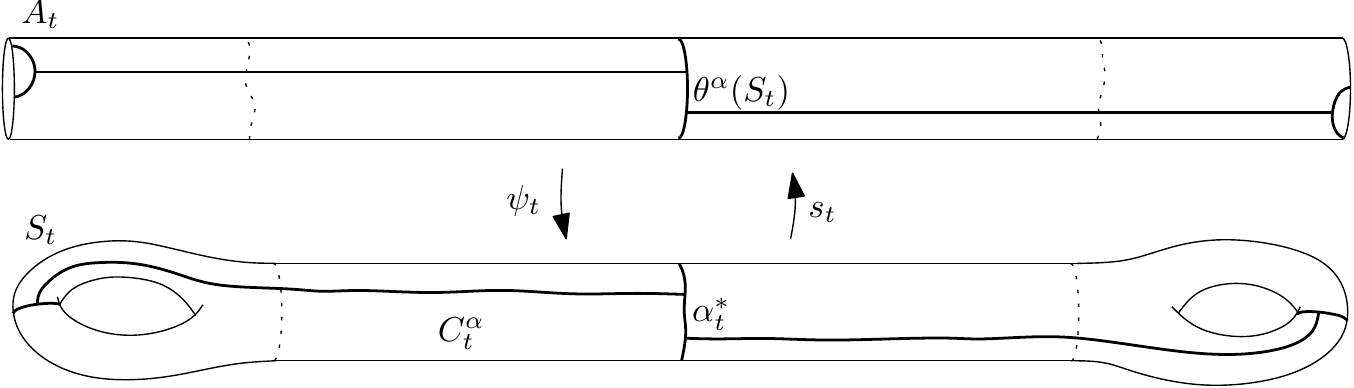}
\caption{The covering annulus $A_t \to S_t$ corresponding to $\alpha_t$.}
\end{figure}

By Lemma \ref{outside collar} and Lemma \ref{straightsubcylinder}, there exists some $T$ such that if $t\geq T$, then the geodesic $\alpha_t^*$ is between $B_t^-$ and $B_t^+$. Let $b_t^\pm$ be the last point of intersection of $\eta_t^\pm$ with $B_t^\pm$, and let $b^\pm$ be the last point of intersection of $\eta^\pm$ with $B^\pm$. By Lemma \ref{orthocompact}, we have $g_t(b_t^\pm) \to b^\pm$ as $t \to \infty^E$. Thus we can choose $T$ large enough so that if $t \geq T$, the distance between $g_t(b_t^\pm)$ and $b^\pm$ is at most $\varepsilon$.

We can now show that if $t \geq T$, then $\theta^\alpha(S_t)$ is within $6\varepsilon$ of $\theta_\infty^+ - \theta_\infty^-$. Let $z_t^\pm$ denote the feet of $\eta_t^\pm$ on $\alpha_t^*$. The half-twist $\theta^\alpha(S_t)$ is equal to the angular distance from $s_t(z_t^-)$ to $s_t(z_t^+)$ in $A_t$. In turn, this is also equal to the angular distance from $s_t(b_t^-)$ to $s_t(b_t^+)$ since the seams $s_t(\eta_t^-)$ and $s_t(\eta_t^+)$ are longitudinal in $A_t$. As $b_t^\pm$ is on the boundary of $C_t^\alpha \setminus B_t^+ \cup B_t^-$ on which $s_t$ is $\varepsilon$-close to an isometry, the latter distance differs by at most $2\varepsilon$ from the angular distance from $b_t^-$ to $b_t^+$ in $C_t^\alpha$. The angular coordinate of $b_t^\pm$ in $C_t^\alpha$ is equal to the angular coordinate of $g_t(b_t^\pm)$ in $C^\pm$, and the latter differs from the angular coordinate of $b^\pm$ by at most $\varepsilon$. Thus $\theta^\alpha_t$ is within $4 \varepsilon$ of the angular distance from $b^-$ to $b^+$. Finally, recall that the angular coordinate of $\eta^\pm$ does not change by more than $\varepsilon$ past the point $b^\pm$, so that the difference $\theta_\infty^+ - \theta_\infty^-$ between the limiting angular coordinates of $\eta^+$ and $\eta^-$ is at most $6 \varepsilon$ away from $\theta^\alpha(S_t)$. 
\end{proof}

This also proves Theorem \ref{twists converge}, as we observed in Section 2.

\begin{acknowledgements}
The author thanks Jeremy Kahn for suggesting this problem and an outline of its solution. Thanks to Joe Adams and Chandrika Sadanand for listening to some parts of the proof, and to Paul Carter for reading an early draft.
\end{acknowledgements}

\bibliographystyle{amsalpha}

\begin{thebibliography}{99}

\bibitem[Ahl]{Ahlfors}
L.V. Ahlfors, \emph{Conformal invariants: topics in geometric function theory},
  AMS Chelsea Publishing, 1973.

\bibitem[DH]{DouadyHubbard}
A.~Douady and J.H. Hubbard, \emph{A proof of {T}hurston's topological
  characterization of rational functions}, Acta Math. \textbf{171}
  (1993), 263--297.

\bibitem[Dum]{Dumas}
D.~Dumas, \emph{Complex projective structures}, Handbook of {T}eichm{\"u}ller
  Theory (A.~Papadopoulos, ed.), vol.~2, European Mathematical Society, 2009.

\bibitem[Dur]{Duren}
P.L. Duren, \emph{Univalent functions},
  Springer-Verlag, New York, 1983.

\bibitem[Hen]{Hensel}
S.W. Hensel, \emph{Iterated grafting and holonomy lifts of {T}eichm{\"u}ller space}, Geom. Dedicata, \textbf{155} (2011), 31--67.

\bibitem[Mas]{Masur}
H.~Masur, \emph{On a class of geodesics in {T}eichm{\"u}ller space}, Ann. of
  Math. \textbf{102} (1975), 205--221.
  
\bibitem[McM]{McMullen}
C.T.~McMullen, \emph{Complex dynamics and renormalization}, Princeton University Press, 1994.
  
\bibitem[Wol]{Wolpert}
S.~Wolpert, \emph{Weil-Petersson perspectives}, Problems on Mapping Class Groups and Related Topics (B.~Farb, ed.), American Mathematical Society, 2006.

\end{thebibliography}

\end{document}